\documentclass[a4paper,reqno,11pt]{amsart}
\usepackage{amsmath}
\usepackage{amssymb}
\usepackage{mathrsfs}
\usepackage[all,cmtip]{xy}
\usepackage{cite}
\usepackage{cases}
\usepackage{enumitem}
\usepackage{indentfirst}
\usepackage{hyperref}
\allowdisplaybreaks[4]

\newtheorem{thm}{Theorem}[section]
\newtheorem{prop}[thm]{Proposition}

\newtheorem{lem}[thm]{Lemma}

\newtheorem{cor}[thm]{Corollary}

\theoremstyle{definition}
\newtheorem{definition}[thm]{Definition}
\newtheorem{example}[thm]{Example}

\theoremstyle{remark}
\newtheorem{remark}[thm]{Remark}

\numberwithin{equation}{section}

\newcommand{\bQ}{\mathbb{Q}}

\newcommand{\bN}{\mathbb{N}}
\newcommand{\bZ}{\mathbb{Z}}

\newcommand{\bP}{\mathbb{P}}

\newcommand\OO{{\mathcal{O}}}

\newcommand\mld{\text{\rm mld}}

\newcommand\XX{{X^{\prime}}}
\newcommand{\mm}{m^{\prime}}
\newcommand{\pX}{{\pi^*(K_X)}}

\newcommand{\pH}{\pi^*(H)}

\newcommand{\mult}{\operatorname{mult}}

\newcommand{\Supp}{\operatorname{Supp}}

\newcommand{\vol}{\operatorname{vol}}

\usepackage{todonotes}

\begin{document}

\title{ON EXPLICIT BIRATIONAL GEOMETRY FOR POLARISED VARIETIES}
\date{\today}

\author{MINZHE ZHU}
\address{MINZHE ZHU, School of Mathematical Sciences, Fudan University, Shanghai, 200433, China}
\email{zhumz20@fudan.edu.cn}


\begin{abstract}
In this paper, we investigate the explicit birational geometry for projective $\epsilon$-lc varieties polarised by nef and big Weil divisors. We show that if $X$ is a projective $\epsilon$-lc variety, $H$ is a nef and big Weil divisor with $\dim\overline{\varphi_{H}(X)}\geq n-1$ and $L$ is an effective Weil divisor such that $|L-K_X|\neq \emptyset$ or $L-K_X$ is nef, then we can find an explicit lower bound of $\vol(H)$ and prove that $|L+\mm H|$ is birational for $\mm\geq m$, where $m$ is an explicit number which depends only on $n$ and $\epsilon$. This result can be applied to polarised Calabi-Yau varieties, Fano varieties and varieties of general type, generalizing the results in \cite{CEW22} and \cite{ZMZ23a}.
\end{abstract}

\maketitle
\pagestyle{myheadings} \markboth{\hfill M.Z.~Zhu
\hfill}{\hfill On explicit birational geometry for polarised varieties \hfill}

\section{Introduction}
Throughout this article, we work over the field of complex numbers $\mathbb{C}$.

Given a projective $\epsilon$-lc variety $X$ of dimension $n$ and a nef and big Weil divisor $H$, there are two central questions in birational geometry:
\begin{enumerate}
\item[(i)] the universal lower bound of $\vol(H)$;
\item[(ii)] the universal $m$ such that $|mH|$ and $|K_X+mH|$ are birational. 
\end{enumerate}
Especially, the following three cases are the most interesting:
\begin{enumerate}
\item $H=K_X$ is nef and big, i.e. $X$ is a variety of general type;
\item $H=-K_X$ is nef and big, i.e. $X$ is a weak Fano variety;
\item $K_X\equiv0$ and $H$ is  nef and big, i.e. $(X,H)$ is a polarised Calabi-Yau variety.
\end{enumerate}

When $n\leq3$ and $X$ has at worst terminal singularities, there are a lot of research relating to the above three cases; see for example \cite{1973Canonical,Reider1988VectorBO} for surface case,\cite{CC08Fano,Jiang16,Chen18ON,JZFano3} for threefold case.

As for arbitrary dimensions, given that $H-K_X$ is pseudo-effective, the existence of the universal lower bound of volume and the universal birationality of polarised varieties are guaranteed by the following remarkable theorem given by Birkar \cite{Birkarpolarised}.

\begin{thm}\cite[Theorem 1.1]{Birkarpolarised}
Let $d$ be a natural number and $\epsilon$ be a positive real number. Then there exists a natural number $m$ depending only on $d,\epsilon$ satisfying the following. Assume that 
\begin{itemize}
\item $X$ is a projective $\epsilon$-lc variety of dimension $d$,
\item $H$ is a nef and big integral divisor on $X$, and 
\item $H-K_X$ is pseudo-effective.
\end{itemize}
Then $|m^\prime H+L|$ and $|K_X+m^\prime H+L|$ define birational maps for any natural number $m^\prime\geq m$ and any integral pseudo-effective divisor $L$.
\end{thm}

However, Birkar used potential birationality of divisors and properties of bounded family in \cite{Birkarpolarised} to prove this theorem, which is not a practical way to calculate the explicit bounds of volume and birationality. Hence it is natural to ask under what conditions can we find the explicit universal bounds of volume and birationality.

The first result was given by Chen-Esser-Wang. In \cite{CEW22}, they gave the optimal bounds of volume and canonical stability for minimal projective $n$-folds of general type with canonical dimension $n$ or $n-1$. Later, in \cite{ZMZ23a} the author used a similar method to give the optimal bounds for weak Fano varieties and polarised Calabi-Yau varieties.

However, these results only applied to varieties with at worst canonical singularities, and the canonical divisor has strict numerical restriction.

In this article, we will estimate the explicit bounds of volume and birationality for a general polarised variety $(X,H)$ with $\epsilon$-lc singularities and
$\dim \overline{\varphi_H(X)}\geq \dim X-1$.

The main results of this paper are the following:
\begin{thm}\label{mainthm1}
Let $X$ be a projective $\epsilon$-lc variety of dimension $n$. Let $H$ be a nef and big Weil divisor on $X$ and $L$ be an effective Weil divisor satisfying one of the following conditions:
\begin{enumerate}
\item[(i)] $|L-K_X|\neq\emptyset$;
\item[(ii)]$L-K_X$ is nef.
\end{enumerate}

Then each of the following conclusions holds:
\begin{enumerate}
\item If $\dim\overline{\varphi_{H}(X)}= n$ and $\varphi_H$ is not birational, then ${\rm vol}(H)\geq 2$ and $|L+mH|$ is birational for $m\geq n+1$;
\item If $\dim\overline{\varphi_{H}(X)}=n-1$,
then $\vol(H)\geq \frac{\epsilon}{n}$ and $|L+mH|$ is birational for $m\geq n+\lfloor\frac{2n}{\epsilon}\rfloor$.
\end{enumerate}
\end{thm}

\begin{remark}
We can use this theorem to prove birationality as following: Let $H$ be a nef and big Weil divisor on a projective $\epsilon$-lc variety $X$ of dimension $n$ and we want to know when $|mH|$ or $|K_X+mH|$ will be birational. If we know that $|tH|$ induces a rational map from $X$ to an image of dimension $\geq n-1$, then we consider the following two special cases: If $|K_X+H|\neq \emptyset$, then we take the effective divisor $L\in |K_X+H|$ and conclude that $|K_X+(1+mt)H|$ is birational for $m\geq n+\lfloor \frac{2n}{\epsilon}\rfloor$; If $tH-K_X$ is nef, then we take the effective divisor $L\in|tH|$ and conclude that $|mtH|$ is birational for $m\geq n+\lfloor \frac{2n}{\epsilon}\rfloor+1$. (Roughly speaking, from the image of dimension $\geq n-1$ to birational image, we only need to increase the divisor by $n+\lfloor \frac{2n}{\epsilon}\rfloor+1$ times.)
\end{remark}

We sketch the proof of Theorem \ref{mainthm1} here. Given a projective $\epsilon$-lc variety of dimension $n$, a nef and big Weil divisor $H$ with $\dim \overline{\varphi_H(X)}\geq n-1$, and an effective Weil divisor $L$ satisfying Condition $(i)$ or $(ii)$. First we take a sufficiently high resolution $\pi:\XX \to X$ such that $\pH=M+N$, where $M$ is the base point free moving part and $N$ is the fixed part. 

The strategy for estimating the volume of $H$ and the birationality of $|L+mH|$ is to slice $\XX$ by general members of $|M|$ for $n-1$ times and get a chain of smooth varieties $Z_1\subseteq Z_2\subseteq \cdots \subseteq Z_n=\XX$ with $\dim Z_i=i$. By basic calculation of intersection number and Birationality Principle \ref{BP}, we can reduce both problems to the curve $Z_1$. The remaining question is to estimate the lower bound of $\deg_{Z_1}\pH$, which is the most difficult part in the proof. We will deal with this tricky part in Subsection \ref{dim=n-1}.
 ~\\

As corollaries, we apply this theorem to the following three cases:
\begin{enumerate}
\item $K_X$ is nef and big and $H=lK_X$ for some $l\in \bN$;
\item $-K_X$ is nef and big and $H=-lK_X$ for some $l\in \bN$; 
\item $K_X\equiv0$ and $H$ is a nef and big Weil divisor on $X$.
\end{enumerate}

\begin{cor}\label{general type}
Let $X$ be a projective $\epsilon$-lc variety of dimension $n$ with $K_X$ nef and big and $l$ be a positive integer.
\begin{enumerate} 
\item If $\dim\overline{\varphi_{lK_X}(X)}= n$ and $\varphi_{lK_X}$ is not birational, then ${\rm vol}(K_X)\geq \frac{2}{l^n}$ and $|mlK_X|$ is birational for $m\geq n+2$; 
\item If $\dim\overline{\varphi_{lK_X}(X)}=n-1$,
then $\vol(K_X)\geq \frac{\epsilon}{nl^n}$ and $|mlK_X|$ is birational for $m\geq n+\lfloor\frac{2n}{\epsilon}\rfloor+1$.
\end{enumerate}
\end{cor}

\begin{cor}\label{Fano}
Let $X$ be a projective $\epsilon$-lc variety of dimension $n$ with $-K_X$ nef and big and $l$ be a positive integer.
\begin{enumerate} 
\item If $\dim\overline{\varphi_{-lK_X}(X)}= n$ and $\varphi_{-lK_X}$ is not birational, then ${\rm vol}(-K_X)\geq \frac{2}{l^n}$ and $|-mlK_X|$ is birational for $m\geq n+1$; 
\item If $\dim\overline{\varphi_{-lK_X}(X)}=n-1$,
then $\vol(-K_X)\geq \frac{\epsilon}{nl^n}$ and $|-mlK_X|$ is birational for $m\geq n+\lfloor\frac{2n}{\epsilon}\rfloor$.
\end{enumerate}
\end{cor}

\begin{cor}\label{CY}
Let $X$ be a projective $\epsilon$-lc variety of dimension $n$ with $K_X\equiv 0$, $H$ be a nef and big Weil divisor on $X$ and $l$ be a positive integer.
\begin{enumerate} 
\item If $\dim\overline{\varphi_{lH}(X)}= n$ and $\varphi_{lH}$ is not birational, then $\vol(H)\geq \frac{2}{l^n}$ and $|mlH|$ is birational for $m\geq n+1$; 
\item If $\dim\overline{\varphi_{lH}(X)}=n-1$,
then $\vol(H)\geq \frac{\epsilon}{nl^n}$ and $|mlH|$ is birational for $m\geq n+\lfloor\frac{2n}{\epsilon}\rfloor$.
\end{enumerate}
\end{cor}

Corollary \ref{general type} generalizes the result in \cite{CEW22}; Corollary \ref{Fano} and Corollary \ref{CY} are generalisations of the results in \cite{ZMZ23a}.
~\\

Another application of Theorem \ref{mainthm1} is to estimate the volume and birationality of a projective $\epsilon$-lc polarised surface. Notice that given a 
projective surface $S$ and a $\bQ$-Cartier Weil divisor $H$, if $h^0(H)\geq2$, then $\dim \overline{\varphi_H(S)}\geq 1$.

\begin{cor}\label{surface}
Let $S$ be a projective $\epsilon$-lc surface. Let $H$ be a nef and big Weil divisor on $S$ with $h^0(H)\geq2$ and $L$ be an effective Weil divisor satisfying one of the following conditions:
\begin{enumerate}
\item[(i)] $|L-K_S|\neq\emptyset$;
\item[(ii)]$L-K_S$ is nef.
\end{enumerate}

Then $\vol(H)\geq \frac{\epsilon}{2}$ and $|L+mH|$ is birational for $m\geq 2+\lfloor\frac{4}{\epsilon}\rfloor$.
\end{cor}

In the last section, we give some examples to show that in Theorem \ref{mainthm1} our estimations are almost optimal. 

\section{Preliminary}
\subsection{Basic definitions} We recall some basic definitions in birational geometry in this subsection.

\begin{definition}[Volume]
Let $X$ be a projective variety of dimension $n$ and $D$ be a $\bQ$-divisor in $X$. Define the  {\it volume} of $D$ as 
\begin{equation*}
    {\rm vol}(D)=\varlimsup_{m\to \infty}\frac{n!h^0(\lfloor mD\rfloor)}{m^n}.
\end{equation*}
If $D$ is nef, then $\vol(D)=D^n$.
\end{definition}

\begin{definition}[Singularities] Let $X$ be a projective  normal variety such that $K_X$ is $\bQ$-Cartier. Let $\XX \to X$ be a resolution of $X$ and $K_{\XX}+B_{\XX}$ be the pullback of $K_X$. The log discrepancy of a prime divisor $D$ on $\XX$ is defined as 
\begin{equation*}
a(D,X,0):=1-\mu_D B_{\XX}.
\end{equation*}
The minimal log discrepancy of $X$ at a point $x$ is defined as
\begin{equation*}
\mld(X\ni x):=\min\{a(D,X,0)|\text{$D$ is a divisor on $\XX$ with center $\{x\}$}\},
\end{equation*}
and the minimal log discrepancy of $X$ is defined as
\begin{equation*}
\mld(X):=\min\{a(D,X,0)|\text{$D$ is a divisor on $\XX$}\},
\end{equation*}
Given $\epsilon\in (0,1]$, we say that $X$ is $\epsilon$-lc if $\mld(X)\geq \epsilon$.
\end{definition}

\begin{definition}[Rational map defined by a Weil divisor]\label{rational map}

Given a projective variety $X$ of dimension $n$ and an effective $\bQ$-Cartier Weil divisor $H$ on $X$ with $h^0(X,H)\geq2$, we define the rational map $\varphi_{H}$ as the following:

By Hironaka’s big theorem, we can take successive blow-ups $\pi:\XX \to X$ such that:
\begin{itemize}
\item $\XX$ is nonsingular projective;
\item $\pi^*(H)=M+N$ where $M$ is the base point free moving part and $N$ is the fixed part. 
\item the union of the support of $\pi^*(H)$ and the exceptional divisors of
$\pi$ is simple normal crossings.
\end{itemize}

Since $|M|$ is base point free, it induces a morphism $\mu: \XX \to T$. Hence we define $\varphi_{H}$ as the induced rational map $\varphi_{H}:X \dashrightarrow T$ and define $\overline{\varphi_{H}(X)}:=T$.
\end{definition}

\subsection{Moving part}
The following lemma is useful to compare the moving part of a  linear system with the counterpart of its restriction.
\begin{lem}\label{moving part}\cite[lemma2.7]{CM01b}
Let X be a smooth projective variety of
dimension $\geq 2$. Let $D$ be a divisor on $X$, $h^0(X,\OO_X(D))\geq 2$ and $S$ be a smooth irreducible divisor on $X$ such that $S$ is not a fixed component of $|D|$. Denote by $M$ the movable part of $|D|$ and by $N$ the movable part of $|D|_S|$ on $S$. Suppose the natural restriction map 
\begin{equation*}
 H^0(X,\OO_X(D))\xrightarrow{\theta}H^0(S,\OO_S(D|_S)) 
\end{equation*}
is surjective. Then $M|_S\geq N$ and thus 
\begin{equation*}
    h^0(S,\OO_S(M|_S))=h^0(S,\OO_S(N))=h^0(S,\OO_S(D|_S)).
\end{equation*}
\end{lem}

\subsection{Projection formula}
We need the following projection lemma to control the moving part of a Cartier divisor in a resolution:
\begin{lem}\label{projection}\cite[Lemma 2.3]{CM11}
Let $X$ be a normal projective variety and $D$ be a $\bQ$-Cartier Weil divisor. Let $\pi:\XX \rightarrow X$ be a resolution of singularities. Assume that $E$ is an effective exceptional $\bQ$-divisor on $\XX$ such that $\pi^*(D)+E$ is a Cartier divisor on $\XX$. 

Then
\begin{equation*}
    \pi_*\OO_{\XX}(\pi^*(D)+E)=\OO_X(D), 
\end{equation*}
where $\OO_X(D)$ is the reflexive sheaf corresponding to the Weil divisor $D$.
\end{lem}

\begin{remark}
Especially, we conclude that $h^0(\XX,\pi^*(D)+E)=h^0(X,D)$, hence $\text{Mov}|\pi^*(D)+E|\leq \pi^*(D)$.
\end{remark}

\section{Birationality Principle}

In this section we introduce a useful method to prove birationality of a linear system.
\begin{definition}\cite[Definition 2.3]{CZ08}
A \textbf{generic irreducible element} $S$ of a movable linear system $|M|$ on a variety $X$ is a generic irreducible component in a general member of $|M|$. 
\end{definition}
\begin{remark}\label{gie}
By definition one can easily see that 
\begin{enumerate}
\item if dim $\overline{\varphi_{|M|}(X)}\geq 2$, then $S$ is a general member of $|M|$;
\item If dim $\overline{\varphi_{|M|}(X)}=1$ (i.e.$|M|$ is composed with a pencil), then $M\equiv tS$ for some integer $t\geq h^0(M)-1$.
\end{enumerate}
\end{remark}

\begin{definition}\cite[Definition 2.6]{Exp3}
Let $|M|$ be a movable linear system on a variety $X$. We say $|M|$ distinguishes two different generic irreducible elements $S_1,S_2$ if $ \overline{\varphi_{|M|}(S_1)}\not= \overline{\varphi_{|M|}(S_2)}$.
\end{definition}

We will frequently use the following birational principle in \cite[Section 2.7]{Exp2} to prove birationality.

\begin{prop}(Birationality Principle)\label{BP}. Let $D$ and $M$ be two divisors on a smooth projective variety $X$. Assume that $|M|$ is base point free. Take the Stein factorization of $\varphi_{|M|}: X\xrightarrow{f} W\rightarrow \bP^{h^0(X,M)-1}$, where $f$ is a fibration onto a normal variety $W$. For a sublinear system $V \subset |D+M|$, the rational map $\varphi_V$ is birational onto its image if 
one of the following conditions satisfies:
\begin{enumerate}
\item $\dim \varphi_{|M|}(X)\geq 2, |D|\not= \emptyset$ and $\varphi_V|_S$ is birational for a general member $S$ of $|M|$;
\item  $\dim \varphi_{|M|}(X)=1$, $\varphi_V$ distinguishes general fibers of $f$ and $\varphi_V|_F$ is birational for a general fiber $F$ of $f$.
\end{enumerate}
\end{prop}

\section{Proof of Theorem \ref{mainthm1}}
In this section we assume that $X$ is a projective $\epsilon$-lc variety of dimension $n$ and $H$ is a nef and big Weil divisor such that $$\dim \overline{\varphi_{H}(X)}=d\geq n-1.$$ 

Step 1: In this step we setup basic notations and make some assumptions.

Taking a small $\bQ$-factorialisation we can assume that $X$ is $\bQ$-factorial. Let $\pi:\XX \to X$ be a sufficiently high resolution such that $\pH=M+N$, where $|M|$ is the base point free movable part, $N$ is the fixed part. Take $Z_n=\XX$.

Inductively, for $2\leq k \leq n$ we can assume $Z_{k-1}$ as a generic irreducible element of $|M|_{Z_k}|$. By Bertini's theorem, we have the following chain of smooth projective subvarieties:
\begin{equation}\label{chain1}
    Z_1\subset \cdots \subset Z_{n-1} \subset Z_n=\XX.
\end{equation}

Since $Z_{k-1}$ is general, we can assume that 
\begin{enumerate}
\item $\pH|_{Z_j}$ is big for each $j=1,2,\cdots n-1$;
\item $\pH|_{Z_i}\equiv Z_{i-1}+N|_{Z_i}$ for $i=3,4,\cdots, n-1$ and $\pH|_{Z_2}\equiv \beta Z_1+N|_{Z_2}$ where $\beta \in \bN$ and $Z_{j-1}\notin \Supp (N|_{Z_j})$ for $j\geq2$.
\end{enumerate}
Replacing $\XX$ we can assume $\Supp(\pH), Z_i$ and exceptional divisors are simple normal crossing.

Step 2: In this step and the next step, we reduce the problem to the curve $Z_1$ and get our key inequality \eqref{eq4.4}, which is essential to estimate the lower bound of $\deg_{Z_1}\pH$ and hence estimate $\vol(H)$ and birationality.

If $m>n-2+\frac{1}{\beta}$, then
$$m\pH-N-Z_{n-1}\equiv(m-1)\pH$$
is a nef and big $\bQ$-divisor with simple normal crossing fractional part.

By Kawamata-Viehweg vanishing theorem\cite{vanish2,vanish1},
 \begin{align*}
 |K_{\XX}+\lceil m\pH\rceil||_{Z_{n-1}}&\succcurlyeq |K_{\XX}+\lceil m\pH-N\rceil||_{Z_{n-1}}\\&\succcurlyeq |K_{Z_{n-1}}+\lceil(m\pH-Z_{n-1}-N)|_{Z_{n-1}}\rceil|
 \end{align*}
By induction, for $i=3,\cdots,n-1$, we have
\begin{align*}
&\indent |K_{\XX}+\lceil m\pH\rceil||_{Z_{i-1}}\\
 &\succcurlyeq |K_{Z_{i}}+
 \lceil m\pH|_{Z_{i}}-\sum_{k=i+1}^n(Z_{k-1}|_{Z_i}+N|_{Z_i})\rceil||_{Z_{i-1}}\\
    &\succcurlyeq |K_{Z_{i-1}}+\lceil m\pH|_{Z_{i-1}}-\sum_{k=i}^n(Z_{k-1}+N|_{Z_{k-1}})|_{Z_{i-1}}\rceil|
 \end{align*}
Hence 
\begin{align*}
    &\indent |K_{\XX}+\lceil m\pH\rceil||_{Z_{1}}\\
 &\succcurlyeq|K_{Z_2}+\lceil m\pH|_{Z_2}-\sum_{k=3}^n(Z_{k-1}+N|_{Z_{k-1}})\rceil||_{Z_1} \\
    &\succcurlyeq |K_{Z_1}+\lceil m\pH|_{Z_1}-\sum_{k=3}^n(Z_{k-1}+N|_{Z_{k-1}})|_{Z_1}-(Z_1|_{Z_1}+\frac{1}{\beta}N|_{Z_1})\rceil|
 \end{align*}
 
 Step 3: Define
 $$M_m:=\text{Mov}|K_{\XX}+\lceil m\pH\rceil|,$$
 $$P_m:=m\pH|_{Z_1}-\sum_{k=3}^n(Z_{k-1}+N|_{Z_{k-1}})|_{Z_1}-(Z_1|_{Z_1}+\frac{1}{\beta}N|_{Z_1}).$$
Since
$$P_m\equiv(m-(n-2+\frac{1}{\beta}))\pH|_{Z_1},$$
we have $\deg_{Z_1}(P_m)=\deg_{Z_1}(m-(n-2+\frac{1}{\beta}))\pH$.

By Lemma \ref{moving part}, 
\begin{equation}\label{eq4.2}
    M_m|_{Z_1}\geq \text{Mov}|K_{Z_1}+\lceil P_m\rceil|
\end{equation}
 for any $m>n-2+\frac{1}{\beta}$.

On the other hand, we can write $K_{\XX}=\pi^*(K_X)+E-F$ where $E,F$ are effective exceptional $\bQ$-divisors with no common components. Therefore,
\begin{equation*}
    \begin{aligned}
    K_{\XX}+\lceil m\pH\rceil&=\pX+E-F+\lceil m\pH\rceil\\&= \pX+m\pH+E+\{-m\pH\}-F.
    \end{aligned}
\end{equation*}
Since $E+\{-m\pH\}$ is an effective exceptional $\bQ$-divisor, by Lemma \ref{projection},
\begin{equation*}
\begin{aligned}
 \pi_*\OO_{\XX}(K_{\XX}+\lceil m\pH\rceil)&\subseteq \pi_*\OO_{\XX}(\pX+m\pH+E+\{-m\pH\})\\&=\OO_X(K_X+mH).
\end{aligned}
\end{equation*}
Hence
\begin{equation}\label{eq4.3}
    M_m\leq \pX+m\pH.
\end{equation}

Combine \eqref{eq4.2},\eqref{eq4.3} and we conclude that 
\begin{equation}\label{eq4.4}
(\pX+m\pH)|_{Z_1}\geq \text{Mov}|K_{Z_1}+\lceil P_m\rceil|.
\end{equation}

In the remaining of this section we divide into two cases, the first case is $\dim\overline{\varphi_{H}(X)}=n$, which is relatively simple. The second case is $\dim\overline{\varphi_{H}(X)}=n-1$, which is much more tricky and take the most length in this section.

\subsection{Case 1: \texorpdfstring{$\dim \overline{\varphi_{H}(X)}=n$}. and \texorpdfstring{$\varphi_H$}. is not birational}\label{dim=n}
In this case, we have $\beta=1$ as in Remark \ref{gie}.

Step 4: In this step we estimate $\deg_{Z_1}\pH$ and $\vol(H)$.

Since $\varphi_H$ is not birational, $\varphi_M$ is not birational, hence $\deg \varphi_M\geq2$ and 
\begin{equation*}
M^n=\deg \varphi_M\cdot\deg \varphi_M(\XX)\geq 2.
\end{equation*}
Since $\pH$ and $M$ are nef, we have the following inequalities:
$$\deg_{Z_1}\pH=(\pH)^{n-1}\cdot Z_1\geq M^{n-1}\cdot Z_1=M^n\geq 2,$$
$$\vol(H)=\vol(\pH)\geq M^n\geq 2.$$

Step 5: In this step we consider the birationality of $|L+mH|$.

Recall that $L$ is an effective divisor. Applying Proposition \ref{BP} on chain \eqref{chain1} inductively implies that
\begin{equation*}
\begin{aligned}
\varphi_{|\pi^*(L+mH)|} \text{ is birational}& \iff  \varphi_{|\pi^*(L+mH)||_{Z_{n-1}}} \text{is birational}\\& \iff \cdots \iff \varphi_{|\pi^*(L+mH)||_{Z_1}} \text{is birational}
\end{aligned}
\end{equation*}

\subsubsection{Case 1.1: $|L-K_X|\neq \emptyset$.}\label{Case1.1}
In this subcase we can modify $L$ so that $L\geq K_X$. 
Since $Z_1$ is general, we can further assume that $\pi^*(L)|_{Z_1}\geq\pX|_{Z_1}$, hence $\varphi_{|\pi^*(L+mH)||_{Z_1}}$ is birational if $\varphi_{|\pi^*(K_X+mH)||_{Z_1}}$ is birational.

By \eqref{eq4.4}, $\varphi_{|\pi^*(K_X+mH)||_{Z_1}}$ is birational if and only if $\varphi_{|K_{Z_1}+\lceil P_{m}\rceil|}$ is birational. This is true when $\deg_{Z_1}P_m=\deg_{Z_1}(m-n+1)\pH>2$. Hence $\varphi_{|L+mH|}$ is birational for $m\geq n+1$.

\subsubsection{Case 1.2: $L-K_X$ is nef}\label{case1.2} 

In this subcase we can modify $X^\prime$ so that $\pi^*(L-K_X)$, $\pH, Z_i$ and exceptional divisors are simple normal crossing. Since $\pi^*(H)|_{Z_j}$ is nef and big for each $j=1,2,\cdots,n-1$ and $\pi^*(L-K_X)$ is nef, $\pi^*(L-K_X+mH)|_{Z_j}$ is nef and big for $j=1,2,\cdots,n-1$.

As the same with Step 2, by Kawamata-Viehweg vanishing theorem \cite{vanish2,vanish1} and induction, we have 
\begin{equation*}
|K_{\XX}+\lceil\pi^*(L-K_X+mH)\rceil||_{Z_1}\succcurlyeq|K_{Z_1}+\lceil \pi^*(L-K_X)|_{Z_1}+P_m\rceil|.
\end{equation*}

By Lemma \ref{moving part}, we have 
\begin{equation*}
\text{Mov}|K_{\XX}+\lceil \pi^*(L-K_X+mH)\rceil||_{Z_1}\geq \text{Mov}|K_{Z_1}+\lceil\pi^*(L-K_X)|_{Z_1}+P_m\rceil|.
\end{equation*}

As in Step 3, we have
\begin{equation}\label{eq4.5}
\text{Mov}|K_{\XX}+\lceil\pi^*(L-K_X+mH)\rceil|\leq \pi^*(L+mH).
\end{equation}

Hence $\varphi_{|\pi^*(L+mH)||_{Z_1}}$ is birational if $\varphi_{|K_{Z_1}+\lceil\pi^*(L-K_X)|_{Z_1}+P_m\rceil}$ is birational. Since
\begin{equation*}
\deg_{Z_1}\pi^*(L-K_X)|_{Z_1}+P_m
\geq \deg_{Z_1}P_m
=\deg_{Z_1}(m-n+1)\pi^*(H),  
\end{equation*}
this is true when $\deg_{Z_1}(m-n+1)\pi^*(H)>2$. Hence $\varphi_{|L+mH|}$ is birational for $m\geq n+1$.

\subsection{Case 2: \texorpdfstring{$\dim \overline{\varphi_{-lK_X}(X)}=n-1$}.}\label{dim=n-1}
The main difficulty in this case is that since $X$ may have singularities worse than canonical singularities, the exceptional divisors with log discrepancy less than $1$ may contribute negativity on $|K_{Z_1}+\lceil P_m\rceil|_{Z_1}|$, hence we cannot bound $\deg_{Z_1}\pH$ from below away from zero by just taking degree on \eqref{eq4.4} as in \cite{CEW22,ZMZ23a}. 

Therefore, we need to consider the components in $\pH$ more explicitly and this is the motivation of the following lemma:

\begin{lem}\label{component}
Let $X$ be a projective variety of dimension $n$ and $H$ be a nef and big $\bQ$-Cartier Weil divisor with $\dim \overline{\varphi_H(X)}=n-1$. Let $\pi:X^\prime \to X$ be a resolution such that $X^\prime$ is smooth and $\pH=M+N$, where $M$ is the base point free moving part  and $N$ is the fixed part. The linear system $|M|$ induces a morphism $\mu:X^\prime\to T$.

If $S$ is an $\pi$-exceptional prime divisor horizontal over $T$ (i.e. $\mu(S)=T$), then $S\subseteq \Supp N$.
\end{lem}

\begin{proof}
Since $S$ is exceptional, $S\cdot \pH^{n-1}=\pi(S)\cdot H^{n-1}=0$ by projection formula. As $S$ is horizontal over $T$, $S\cdot M^{n-1}>0$. Let 
$$t=\max\{k\in \bN|S\cdot \pH^{n-i}\cdot M^{i-1}=0, 1\leq i\leq k\}.$$
Then 
\begin{equation*}
\begin{aligned}
0&=S\cdot \pH^{n-t}\cdot M^{t-1}\\&=S\cdot \pH^{n-t-1}\cdot M^{t-1}\cdot N+S\cdot \pH^{n-t-1}\cdot M^t.      
\end{aligned}
\end{equation*}

By definition of $t$,
$$S\cdot \pH^{n-t-1}\cdot M^t>0,$$
hence we have $$S\cdot \pH^{n-t-1}\cdot M^{t-1}\cdot N<0.$$

If $S\nsubseteq \Supp N$, then $N|_S \cdot \pH|_S^{n-t-1}\cdot M|_S^{t-1}\geq0$ since $N|_S$ is an effective divisor and $\pH|_S,M|_S$ are nef divisors, which is a contradiction.
\end{proof}

Step 4$^\prime$: In this step we estimate $\deg_{Z_1}\pH$ and $\vol(H)$.

The base point free linear system $|M|$ induces a morphism $\mu:\XX\to T$. Taking the Stein factorialisation, we can assume that $\mu$ has connected fiber. Hence $Z_1$ is the general fiber of $\mu$. 

Recall that $K_{\XX}=\pX+E-F$ where $E,F$ are effective exceptional $\bQ$-divisors with no common components. Assume that $F=\sum_{i=1}^la_iF_i$ where $F_1,\cdots,F_r$ are horizontal over $\overline{\varphi_H(X)}$, and $F_{r+1},\cdots,F_l$ are vertical over $\overline{\varphi_H(X)}$. Since $Z_1$ is general, $\Supp F_i \cap Z_1=\emptyset, i\geq r+1$. Hence
$$\Supp F\cap Z_1= \bigcup_{i=1}^r \Supp F_i \cap Z_1=:\{x_1,\cdots,x_c\}.$$ If $\Supp F\cap Z_1=\emptyset$, we let $c=0$.  Since $X$ is $\epsilon$-lc, $a_i\leq 1-\epsilon$. Since $Z_1$ is general, $F_i$ and $Z_1$ are simple normal crossing for $i\leq r$, hence $\mult_{x_i} F|_{Z_1}\leq 1-\epsilon$. Therefore $\deg_{Z_1}F\leq c(1-\epsilon)$.

By Lemma \ref{component}, $F_i\subseteq \Supp N\subseteq \Supp \pH$ for $i\leq r$. We can write $$\Supp \pH \cap Z_1=\{x_1,\cdots,x_c,x_{c+1},\cdots,x_d\}.$$ Since $\pH|_{Z_1}$ is big, $\Supp \pH\cap Z_1\neq \emptyset$, hence $d\geq 1$.

Let $\mult_{x_i}\pH|_{Z_1}=b_i=\frac{p_i}{q_i}$, where $(p_i,q_i)=1, i=1,2,\cdots,d$. Take $Q=\prod_{i=1}^d q_i$.

Recall that for $m\geq n-2+\frac{1}{\beta}$, we have the following inequality \eqref{eq4.4}
\begin{equation*}
    (\pX+m\pH)|_{Z_1}\geq \text{Mov}|K_{Z_1}+\lceil P_m\rceil|.
\end{equation*}

Since $\pH|_{Z_1}$ is big, $\deg_{Z_1}\pH>0$.
Taking $m=\lambda Q+n,\lambda\in \bN, \lambda\gg0$, we have $\deg_{Z_1}P_m=\deg_{Z_1}(m-(n-2+\frac{1}{\beta})\pH)\geq2$, which implies $$\text{Mov}|K_{Z_1}+\lceil P_m\rceil|=|K_{Z_1}+\lceil P_m\rceil|.$$
Taking degree on \eqref{eq4.4}, the left side is 
\begin{equation*}
\begin{aligned}
\deg_{Z_1} K_{Z_1}+\lceil P_m\rceil
&=\deg_{Z_1} ((K_{\XX}+(n-2)M)|_{Z_2}+Z_1)|_{Z_1}+\lceil P_m\rceil
\\&= \deg_{Z_1} \pX+E-F+\lceil (m-n+2-\frac{1}{\beta})\pH\rceil
\\&\geq \deg_{Z_1} \pX-F+\lceil(m-n+1)\pH\rceil
\\&\geq \deg_{Z_1} \pX+\lceil(m-n+1)\pH\rceil-c(1-\epsilon).
\end{aligned}
\end{equation*}

Therefore, we conclude 
\begin{equation}\label{eq4.6}
\deg_{Z_1} m\pH\geq \deg_{Z_1} \lceil(m-n+1)\pH\rceil-c(1-\epsilon)
\end{equation}

We claim that $\deg_{Z_1}\pH\geq \frac{\epsilon}{n}$. If not, since $$\deg_{Z_1}\pH=\sum_{i=1}^d \mult_{x_i}\pH|_{Z_1}=\sum_{i=1}^d b_i,$$ we have  $b_i<\frac{\epsilon}{n}, 1\leq i\leq d$.

Recall that $b_i=\frac{p_i}{q_i}, Q=\prod_{i=1}^d q_i$, hence $Qb_i\in \bZ$. Therefore,
\begin{equation*}
\begin{aligned}
&\indent \deg_{Z_1} \lceil(m-n+1)\pH\rceil-(m-n+1)\pH
\\&=\deg_{Z_1} \lceil(\lambda Q+1)\pH\rceil-(\lambda Q+1)\pH
\\&=\sum_{i=1}^d\lceil(\lambda Q+1)b_i\rceil-(\lambda Q+1)b_i
\\&=\sum_{i=1}^d\lceil b_i\rceil-b_i= d-\deg_{Z_1}\pH
\end{aligned}
\end{equation*}

By \eqref{eq4.6}, $\deg_{Z_1}n\pH\geq d-c(1-\epsilon)\geq \min\{1,\epsilon\}\geq\epsilon$ since $d\geq c$ and $d\geq1$, which is a contradiction. Hence $\deg_{Z_1}\pH\geq \frac{\epsilon}{n}$.

Since $\pH$ and $M$ are nef, 
\begin{equation*}
\vol(H)=\vol(\pH)\geq \pH\cdot M^{n-1}=\beta \deg_{Z_1}\pH\geq \frac{\epsilon}{n}.
\end{equation*}

Step 5$^\prime$: In this step we consider the birationality of $\varphi_{|L+mH|}$. 

Recall that $L$ is effective. Applying Proposition \ref{BP} on chain \eqref{chain1} inductively, $\varphi_{|L+mH|}$ is birational if and only if
\begin{enumerate}
\item[(I)] $\varphi_{|\pi^*(L+mH)||_{Z_2}}$ distinguishes different generic irreducible elements of $\varphi_{|M||_{Z_2}}$;
\item[(II)] $\varphi_{|\pi^*(L+mH)||_{Z_1}}$ is birational.
\end{enumerate}

\subsubsection{Case 2.1: $|L-K_X|\neq \emptyset$.}\label{Case2.1}
As in Subsection \ref{Case1.1}, Condition (II) is satisfied when $\deg P_m>2$. Hence it is sufficient to consider Condition (I).

 If $\beta=1$, then it is satisfied since $$|\pi^*(L+mH)||_{Z_2}\succcurlyeq|\pi^*H||_{Z_2}\succcurlyeq|M||_{Z_2}.$$ 
 
 If $\beta\geq 2$, choose two different generic irreducible elements $C_1$,$C_2$ of $|M||_{Z_2}$. $M|_{Z_2}-C_1-C_2\equiv(\beta-2)Z_1$ is nef. Therefore, for $m>n-2$, by Kawamata-Viehweg vanishing theorem \cite{vanish2,vanish1} we conclude 
 \begin{equation*}
\begin{aligned}
&\indent |K_{\XX}+\lceil m\pH\rceil||_{Z_2}\notag
\\&\succcurlyeq|K_{\XX}+\lceil(m-n+1)\pH\rceil+(n-1)M||_{Z_2}
\\&\succcurlyeq|K_{Z_2}+\lceil(m-n+1)\pH|_{Z_2}\rceil+M|_{Z_2}|
\end{aligned}     
 \end{equation*}
and the surjective map:
\begin{equation*}
\begin{aligned}
&\indent H^0(Z_2,K_{Z_2}+\lceil(m-n+1)\pH|_{Z_2}\rceil+M|_{Z_2})
\\&\rightarrow H^0(C_1,K_{C_1}+D_1)\oplus H^0(C_2,K_{C_2}+D_2)
\end{aligned}    
\end{equation*}
where
\begin{equation*}
 \begin{aligned}
 D_i:&=(\lceil(m-n+1)\pH|_{Z_2}\rceil+M|_{Z_2}-C_i)|_{C_i}\\&=\lceil(m-n+1)\pH|_{Z_2}\rceil|_{C_i}
\end{aligned}   
\end{equation*}
for $i=1,2$.

If 
$$(m-n+1)\deg_{C_i} \pH=(m-n+1)\deg_{Z_1} \pH>1,$$ then $\deg_{C_i}D_i\geq 2$, which implies $H^0(C_i,K_{C_i}+D_i)\not= 0$, hence $|K_{\XX}+\lceil m\pH\rceil||_{Z_2}$ can distinguish different generic irreducible elements of $|M||_{Z_2}$. By \eqref{eq4.3}, $|\pi^*(K_X+mH)||_{Z_2}$ can also distinguish different generic irreducible elements of $|M||_{Z_2}$. 

Since $|L-K_X|\neq \emptyset$, modifying $L$ we can assume $L\geq K_X$. Since $Z_2$ is a general member of $|M|_{Z_3}|$, $\pi^*(L)|_{Z_2}\geq \pX|_{Z_2}$, hence $|\pi^*(L+mH)||_{Z_2}$ can also distinguish different generic irreducible elements of $|M||_{Z_2}$. Therefore, Condition (I) is satisfied in this case.

In summary, if 
$$\deg_{Z_1} P_m=(m-(n-2+\frac{1}{\beta}))\deg_{Z_1}\pH>2,$$
$$(m-n+1)\deg_{Z_1}\pH>1,$$ then $\varphi_{|L+mH|}$ is birational. Since $\deg_{Z_1}\pH\geq \frac{\epsilon}{n}$ and $\beta\geq1$, both inequalities are satisfied when $m\geq n+\lfloor \frac{2n}{\epsilon}\rfloor$.

\subsubsection{Case 2.2: $L-K_X$ is nef.} As in Subsection \ref{case1.2}, Condition (II) is satisfied when $\deg_{Z_1} P_m>2$. Hence it is sufficient to consider Condition (I).

If $\beta=1$, then it is satisfied since $$|\pi^*(L+mH)||_{Z_2}\succcurlyeq|\pH||_{Z_2}\succcurlyeq|M||_{Z_2}.$$

If $\beta\geq 2$, choose two different generic irreducible elements $C_1,C_2$ of $|M||_{Z_2}$. As in Subsection \ref{Case2.1}, by Kawamata-Viehweg vanishing theorem \cite{vanish2,vanish1} we conclude
\begin{equation*}
\begin{aligned}
&\indent|K_{\XX}+\lceil\pi^*(L-K_X+mH)\rceil||_{Z_2}
\\&\succcurlyeq|K_{\XX}+\lceil\pi^*(L-K_X+(m-n+1)H))\rceil+(n-1)M||_{Z_2}
\\&\succcurlyeq|K_{Z_2}+\lceil\pi^*(L-K_X+(m-n+1)H))|_{Z_2}\rceil+M|_{Z_2}|
\end{aligned}
\end{equation*}
and the surjective map:
\begin{equation*}
\begin{aligned}
&\indent H^0(Z_2,K_{Z_2}+\lceil\pi^*(L-K_X+(m-n+1)H)|_{Z_2}\rceil+M|_{Z_2})
\\&\rightarrow H^0(C_1,K_{C_1}+D^\prime_1)\oplus H^0(C_2,K_{C_2}+D^\prime_2)
\end{aligned}    
\end{equation*}
where
\begin{equation*}
 \begin{aligned}
 D^\prime_i:&=(\lceil\pi^*(L-K_X+(m-n+1)H)|_{Z_2}\rceil+M|_{Z_2}-C_i)|_{C_i}\\&=\lceil\pi^*(L-K_X+(m-n+1)H)|_{Z_2}\rceil|_{C_i}
\end{aligned}   
\end{equation*}
for $i=1,2$.

Since $L-K_X$ is nef, $\deg_{C_i}\pi^*(L-K_X)\geq0$. If $$(m-n+1)\deg_{C_i}\pH=(m-n+1)\deg_{Z_1}\pH>1,$$
then $\deg_{C_i}D^\prime_i\geq 2$, which implies $H^0(C_i,K_{C_i}+D^\prime_i)\neq 0$, hence $|K_{\XX}+\lceil \pi^*(L-K_X+mH)\rceil||_{Z_2}$ can distinguish different generic irreducible elements of $|M||_{Z_2}$.

As in Step 3 inequality \eqref{eq4.3}, we have
\begin{equation*}
\text{Mov}|K_{\XX}+\lceil\pi^*(L-K_X+mH)\rceil|\leq \pi^*(L+mH).
\end{equation*}

Hence $|\pi^*(L+mH)||_{Z_2}$ can also distinguish different generic irreducible elements of $|M||_{Z_2}$. Therefore, Condition (I) is satisfied in this case.

In summary, if 
$$\deg_{Z_1} P_m=(m-(n-2+\frac{1}{\beta}))\deg_{Z_1}\pH>2,$$
$$(m-n+1)\deg_{Z_1}\pH>1,$$ then $\varphi_{|L+mH|}$ is birational. Since $\deg_{Z_1}\pH\geq \frac{\epsilon}{n}$ and $\beta\geq1$, both inequalities are satisfied when $m\geq n+\lfloor \frac{2n}{\epsilon}\rfloor$.

\section{Proof of Corollaries}
In this section we prove the corollaries.
\begin{proof}[Proof of Corollary \ref{general type}]
We take $H=lK_X$ and $L\in |lK_X|$, hence $H$ is nef and big, $L$ is effective and $L-K_X\equiv (l-1)K_X$ is nef. The corollary follows from Theorem \ref{mainthm1}.
\end{proof}

\begin{proof}[Proof of Corollary \ref{Fano}]
We take $H=-lK_X$ and $L=0$, hence $H$ is nef and big,  and $L-K_X=-K_X$ is nef. The corollary follows from Theorem \ref{mainthm1}.
\end{proof}

\begin{proof}[Proof of Corollary \ref{CY}]
In Theorem \ref{mainthm1}, we replace $H$ by $lH$ and take $L=0$, Since $lH$ is nef and big and $L-K_X\equiv0$, The corollary follows from Theorem \ref{mainthm1}.
\end{proof}

\begin{proof}[Proof of Corollary \ref{surface}]
Since $h^0(H)\geq 2$, by Definition \ref{rational map}, $\dim \overline{\varphi_H(S)}\geq1$, hence the corollary follows from Theorem \ref{mainthm1}.
\end{proof}

\section{Examples}\label{example}

In this section we give some examples to show that in Theorem \ref{mainthm1} our estimations are almost optimal.

Since our examples are hypersurfaces in weighted projective spaces, the involved singularities are cyclic quotient singularities. Hence we need the following basic lemma to compute the minimal log discrepancy of a cyclic quotient singularity. For a proof, we recommend the readers to \cite{Ambrotoricmld}.

\begin{lem}\label{toric mld}Let $(x\in X)=\frac{1}{r}(a_1,\cdots, a_n)$ be a cyclic quotient singularity, then
\begin{equation*}
\mld(X\ni x)=\min_{1\leq j\leq r}\sum_{i=1}^n(1+\frac{ja_i}{r}-\lceil\frac{ja_i}{r}\rceil).
\end{equation*}
\end{lem}

\begin{example}
Given $n\geq2$, consider the general hypersurface
$$X=V_{2n+2}\subset \bP(1^{(n+1)},n+1).$$
$X$ is smooth and $\omega_X\cong \OO_X$. Let $L=0$ and $H\cong\OO_X(1)$. Then $\dim\overline{\varphi_{H}(X)}= n$ and $\varphi_H$ is not birational. We have $\vol(H)=2$ and $|mH|$ is birational for $m\geq n+1$.
\end{example}

\begin{example}
Given $N\geq n\geq2$, consider the general hypersurface 
$$X=V_{6N}\subset \bP(1^{(n)},2N,3N).$$
$X$ has cyclic quotient singularities of type $\frac{1}{N}(1^{(n)})$. Hence by Lemma \ref{toric mld}, $\mld(X)=\frac{n}{N}$. Denote $\epsilon=\frac{n}{N}$. Then $X$ is $\epsilon$-lc. Since $\omega_X\cong \OO_X(N-n)$, we can take the effective divisor $L\in |\OO_X(N-n)|$. Let $H\cong \OO_X(1)$. We have $\vol(H)=\frac{1}{N}=\frac{\epsilon}{n}$ and $|L+mH|$ is birational for $m\geq n+2N=n+\frac{2n}{\epsilon}$.
\end{example}

\section*{Acknowledgments} The author expresses his gratitude to his advisor Professor Meng Chen for his great support and encouragement. The author would like to thank Mengchu Li for correcting errors in the first version of this article. The author appreciates Wentao Chang, Hexu Liu, Mengchu Li and Yu Zou for useful discussions.

\bibliographystyle{alpha}
\bibliography{ref}

\end{document}